\documentclass[11pt,reqno]{amsart}
\usepackage{amscd,graphicx,psfrag,amsxtra}
\usepackage{pinlabel}
\usepackage{hyperref}
\usepackage{xcolor}

\newcommand{\tw}{\tilde{\xi}}  

\newcommand{\ii}{\mathbf{I}\,}  

\newcommand{\p}{\partial}

\newcommand{\Z}{\mathbb Z}

\renewcommand{\P}{\mathcal P}

\newcommand{\D}{\mathcal D}

\newcommand{\R}{\mathbb R}

\renewcommand{\phi}{\varphi}

\newcommand{\Tr}{\operatorname{Tr}\,}

\newcommand{\ind}{\operatorname{ind}}

\newcommand{\spin}{\,\operatorname{spin}}

\newcommand{\vol}{\operatorname{vol}}

\newcommand{\zp}{Z_{\infty}}

\newcommand{\conn}{\mathbin{\#}} 

\newcommand{\spinc}{\ifmmode{\operatorname{Spin}^c}\else{$\operatorname{spin}^c$\ }\fi}

\newcommand{\diff}{\operatorname{Diff}}          
\newcommand\rep{\alpha}   
\newcommand{\mm}{\mathfrak M}
\newcommand{\mmp}{\mm^+}

\newcommand\ddfrac[2]{\frac{\displaystyle #1}{\displaystyle #2}}

\newtheorem{theorem}{Theorem}[section]

\newtheorem{lemma}[theorem]{Lemma}
\newtheorem{proposition}[theorem]{Proposition}

\newtheorem{thm}{Theorem}

\theoremstyle{definition}

\newtheorem{remark}[theorem]{Remark}

\title{On positive scalar curvature cobordisms and the conformal Laplacian on end-periodic manifolds}
\thanks{The first author was partially supported by NSF Grant DMS-1547145, the second author was partially supported by NSF Grants DMS-1506328 and DMS-1811111, and the third author was partially supported by a grant from the Simons Foundation \#\,426269.}
\author[Demetre Kazaras]{Demetre Kazaras}
\address{Department of Mathematics\newline\indent Stony Brook University, \newline\indent Stony Brook, NY 11794}
\email{\rm{demetre.kazaras@stonybrook.edu}}
\author[Daniel Ruberman]{Daniel Ruberman}
\address{Department of Mathematics, MS 050\newline\indent Brandeis
University \newline\indent Waltham, MA 02454}
\email{\rm{ruberman@brandeis.edu}}
\author[Nikolai Saveliev]{Nikolai Saveliev}
\address{Department of Mathematics\newline\indent
University of Miami \newline\indent PO Box 249085
\newline\indent Coral Gables, FL 33124}
\email{\rm{saveliev@math.miami.edu}}

\begin{document}
\begin{abstract}
We show that the periodic $\eta$-invariant of Mrowka, Ruberman and Saveliev provides an obstruction to the existence of cobordisms with positive scalar curvature metrics between  manifolds of dimensions $4$ and $6$. Our proof combines the end-periodic index theorem with a relative version of the Schoen--Yau minimal surface technique. As a result, we show that the bordism groups $\Omega^{\spin,+}_{n+1}(S^1 \times BG)$ are infinite for any non-trivial group $G$ which is the fundamental group of a spin spherical space form of dimension $n=3$ or $5$.\end{abstract}

\maketitle

\section{Introduction}
A classic problem in global differential geometry is to determine when a given manifold admits a metric of positive scalar curvature (`psc') and if so, to say something about the classification of such metrics. Fundamental tools for this study make use of index theory \cite{lichnerowicz:spinors} and minimal submanifolds~\cite{SY79}, with many extensions of these methods over the years; see for instance the surveys~\cite{rosenberg:psc-progress,schick:psc-survey}.  Because the space $\P(X)$ of psc metrics on a manifold $X$ is (if non-empty) infinite dimensional, classification {\em per se} is not reasonable, and one considers such metrics up to deformation, or isotopy. This leads to the study of the homotopy groups of $\P(X)$ and of the associated moduli space $\mmp(X) = \P(X)/\diff(X)$, and index theoretic methods show that these homotopy groups can be non-trivial~\cite{hitchin:spinors}.

Botvinnik and Gilkey~\cite{botvinnik-gilkey:bordism,botvinnik-gilkey:spaceforms} used the Atiyah--Patodi--Singer index theorem~\cite{aps:II} to show that, for any odd-dimensional non-simply connected spherical space form $Y$ of dimension at least $5$, the moduli space $\mmp(Y)$ has infinitely many path components. They further strengthened this result by introducing a psc spin bordism group $\Omega^{\spin,+}_n (BG)$, whose definition will be recalled in Section \ref{S:bordisms}. They showed that the group $\Omega^{\spin,+}_n (BG)$ is infinite for every odd $n \ge 5$ and $G$ the fundamental group of a non-simply connected spherical space form of dimension $n$. They also extended these results to 
some non-orientable even-dimensional manifolds. 

The second and third authors, together with T.~Mrowka, extended~\cite{MRS3} the Atiyah--Patodi--Singer theorem to the setting of even-dimensional manifolds with periodic ends. As a consequence, the {\em isotopy} results of~\cite{botvinnik-gilkey:bordism,botvinnik-gilkey:spaceforms} continue to hold for many even dimensional orientable manifolds, for instance, the product of a spherical space form with a circle.  (The results in dimension $4$ are somewhat more limited, because~\cite{marques:psc} psc metrics are unique up to isotopy in dimension $3$). It is natural to ask if these results on classification up to isotopy actually hold up to bordism. 

In this paper, we combine the end-periodic index theory~\cite{MRS3} with the relative version of the Schoen--Yau minimal surface technique due to Botvinnik--Kazaras~\cite{BK} to show that the Botvinnik--Gilkey {\em bordism} results continue to hold for orientable manifolds in low even dimensions. The statement of our result will include a group-theoretic constant $r_n(G)$ defined in~\cite{botvinnik-gilkey:spaceforms} via the representation theory of a finite group $G$. 

\begin{thm}\label{T:one}
Let $n = 4$ or $6$. Let $G$ be a finite group and set $\Gamma=\mathbb{Z}\times G$.  When $n=4$, additionally assume that $G$ is the fundamental group of a non-simply connected $3$-dimensional spherical space form, while if $n=6$ we assume that $r_5(G) > 0$. Then $\Omega^{\spin,+}_n (B\Gamma)$ contains infinitely many elements, represented by maps $f:M\to B\Gamma$ where $M$ are connected manifolds with $\pi_1(M) \cong \Gamma$ that support psc metrics. 
\end{thm}

Following the remark after~\cite[Theorem 0.1]{botvinnik-gilkey:spaceforms}, $r_5(G) > 0$ if $G$ contains an element $g$ which is not conjugate to $g^{-1}$. This holds for example in any odd order group.

It is worth noting that the distinct bordism classes in $\Omega^{\spin,+}_6 (B\Gamma)$ are obtained by constructing different psc metrics on the same underlying smooth manifold. In dimension $4$, the conclusion is a bit weaker: given an integer $N$, there is a $4$-manifold that supports at least $N$ bordism classes of psc metrics. The restriction to relatively low dimensions in Theorem \ref{T:one} has to do with possible singularities in an area-minimizing hypersurface in dimensions greater than $7$; it is conceivable that this restriction could be removed using techniques recently developed in~\cite{lohkamp:cones,schoen-yau:higherdim}. Theorem \ref{T:one} recovers an earlier result of Leichtnam--Piazza \cite[Theorem 0.1]{LeichtnamPiazza:eta} in the case of $n=6$, but is new for $n=4$. The approach in \cite{LeichtnamPiazza:eta} utilizes the higher eta-invariants, see \cite{Lott:eta}, and appears quite different from ours.

Theorem \ref {T:one} is a consequence of the following result about the periodic $\tw$--invariants which were introduced in \cite{MRS3}.

\begin{thm}\label{T:two}
Let $X_0$ and $X_1$ be closed oriented Riemannian spin manifolds of dimension $n = 4$ or $6$ with positive scalar curvature and a choice of  primitive cohomology classes $\gamma_0 \in H^1 (X_0;\Z)$ and $\gamma_1 \in H^1 (X_1;\Z)$ and unitary representations $\alpha_0: \pi_1 (X_0) \to U(k)$ and $\alpha_1: \pi_1 (X_1) \to U(k)$. Suppose that $X_0$ is bordant to $X_1$ via a positive scalar curvature cobordism and that both the cohomology classes $\gamma_0$, $\gamma_1$ and the representations $\alpha_0$ and $\alpha_1$ extend to this cobordism. Then 
\[
\tw_{\alpha_0} (X_0,\D^+) = \tw_{\alpha_1} (X_1,\D^+).
\]
\end{thm}
We prove Theorem \ref{T:two} by applying the index theorem of \cite{MRS3} to an end-periodic manifold $Z_\infty$ constructed from the psc-cobordism between $X_0$ and $X_1$. The manifold $Z_{\infty}$ has two periodic ends modeled on the infinite cyclic covers of $X_0$ and $X_1$ with metrics conformally equivalent to end-periodic metrics. The `middle portion' of $Z_\infty$ is, roughly speaking, a minimal hypersurface with free boundary as constructed in Botvinnik--Kazaras \cite{BK}; this is the crucial geometric ingredient of the proof. Though this middle portion comes equipped with a psc metric, it does not smoothly glue to the given metrics on the covers of $X_0$ and $X_1$. To produce a smooth metric on $Z_\infty$, we introduce a transition region which may initially have negative scalar curvature. The main analytic result of the paper, Lemma \ref{lem:main}, conformally changes this initial metric on $Z_\infty$ to one of positive scalar curvature, without dramatically disturbing it far away from the transition region.

The paper is organized as follows. In Section \ref{S:periodic},  we recall the definition of the invariants $\tw_{\alpha}$ in all even dimensions, together with the formula of \cite{MRS3} expressing $\tw_{\alpha}$ in terms of the Dirac index on end-periodic manifolds. We further extend this formula to a class of end-periodic manifolds with metrics which are conformally equivalent to end-periodic metrics but are not end-periodic themselves. Theorem \ref{T:two} is proved in Section \ref{S:bordisms} assuming Lemma   \ref{lem:main}. The proof of Lemma \ref{lem:main} then occupies the entire Section \ref{S:laplace}. Finally, Theorem \ref{T:one} is proved in Section \ref{S:reps}.\\[2ex]
{\bf Acknowledgments:} The broad outline of this project, encompassing both \cite{BK} and the present paper, took shape during conversations with Boris Botvinnik during the 2015 PIMS Symposium on Geometry and Topology of Manifolds. We thank Boris for his role in the project, and the organizers of the PIMS Symposium for providing a stimulating environment. We are also grateful to the anonymous referees whose comments helped improve the paper.


\section{Periodic $\tw$--invariants}\label{S:periodic}
Let $X$ be a compact even-dimensional spin manifold with a choice of a primitive cohomology class $\gamma \in H^1 (X;\Z)$ and a Riemannian metric $g$ of positive scalar curvature. Associated with this data is the spin Dirac operator $\D^+ = \D^+(X,g)$ and, given a  representation $\rep: \pi_1 (X) \to U(k)$, the twisted Dirac operator $\D^+_{\rep} = \D^+_{\rep}(X,g)$. The periodic $\tw$--invariant was defined in \cite[Section 8.1]{MRS3} by the formula
\[
\tw_{\rep}(X, \D^+)\;=\;\frac 1 2 \left(\eta (X,\D^+_{\rep}) - k\cdot \eta (X,\D^+)\right)
\]
using the periodic $\eta$--invariants of \cite{MRS3}. Since the metric $g$ has positive scalar curvature, the twisted Dirac operators $\D^{\pm}_z = \D^{\pm} - \ln z\cdot df$ have zero kernels on the unit circle $|z| = 1$, and 

\[
\eta (X,\D^+)\;=\;\frac 1 {\pi i}\,\int_0^{\infty} \oint_{|z| = 1} \Tr\,\left(df\cdot \D^+_z e^{-t\, \D^-_z \D^+_z}\right)\,\frac {dz} {z}\;dt.
\]

\medskip\noindent
The definition of $\eta (X,\D^+_{\rep})$ is similar. Of most importance to us is the fact that the periodic $\tw$--invariant can be expressed in index theoretic terms, which is done as follows. 

Let us consider an end-periodic manifold with the end modeled on the infinite cyclic cover $\widetilde X \to X$ corresponding to $\gamma$. To be precise, let $Y\subset X$ be a hypersurface Poincar\'e dual to $\gamma$, and let $W$ be the cobordism from $Y$ to itself obtained by cutting $X$ open along $Y$. We will assume that $Y$ is spin bordant to zero, since this is the only case relevant to this paper, and choose a spin null-cobordism $Z$. The end-periodic manifold in question is then of the form 
\begin{equation}\label{E:zp}
\zp\, =\, Z\, \cup\, W\, \cup\, W\, \cup \ldots
\end{equation}
where we write $\partial W = Y^- \cup\, Y^+$ and identify each $Y^+$ with $Y^-$ in the subsequent copy of $W$, and also identify the boundary of $Z$ with $Y^-$ in the first copy of $W$. Let $f: \zp \to \R$ be any smooth function with the property that $f(\tau(x)) = f(x) + 1$ over the end of $\zp$, where $\tau$ stands for the covering translation. 

Let $g_{\infty}$ be any metric on $\zp$ which matches over the periodic end the lift of the metric $g$ from $X$ to its infinite cyclic cover. Denote by $g_Z$ the induced metric on $Z$. The positive scalar curvature condition then ensures that the Dirac operators $\D^+(\zp,g_\infty)$ and $\D_{\alpha}^+(\zp,g_\infty)$ are uniformly invertible at infinity in the sense of Gromov and Lawson \cite{gromov-lawson:psc} and, in particular, their $L^2$ closures are Fredholm. Of course, the operator $\D_{\alpha}^+(\zp,g_\infty)$ is only well defined if the pull-back of $\alpha$ to $\pi_1(Y)$ extends to a representation of $\pi_1 (Z)$, which we will assume from now on. 

\begin{proposition}\label{P:one}
Let $\zp$ be an end-periodic manifold whose end is modeled on the infinite cyclic cover $\widetilde X$. Then 
\begin{equation}\label{E:eq1}
\tw_{\rep}(X, \D^+)\, =\, k\cdot \ind \D^+ (\zp,g_\infty)\, -\, \ind \D^+_{\alpha} (\zp,g_\infty).
\end{equation}
\end{proposition}

\begin{proof}
Apply the index theorem \cite[Theorem A]{MRS3} to the Dirac operators $\D^+ (\zp,g_\infty)$ and $\D^+_{\alpha} (\zp,g_\infty)$ to obtain
\smallskip
\[
\ind \D^+ (\zp,g_\infty)\, =\, \int_Z \ii (\D^+(Z,g_Z)) - \int_Y\omega\, + \int_X df \wedge\omega\; -\, \frac 1 2\,\eta(X, \D^+),
\]
\[
\ind \D^+_{\alpha} (\zp,g_\infty) = \int_Z \ii (\D^+_{\alpha}(Z,g_Z)) - \int_Y\omega_{\alpha} + \int_X df \wedge\omega_{\alpha} - \frac 1 2\,\eta(X, \D^+_{\alpha}),
\]

\medskip\noindent
where $\ii (\D^+ (Z,g_Z)) = \widehat A\, (Z,g_Z)$ and $\ii (\D_{\alpha}^+(Z,g_Z)) = \widehat A\, (Z,g_Z)\operatorname{ch} (V_{\alpha})$ are the local index forms, and $\omega$ and $\omega_{\alpha}$ are transgressed classes such that $d\omega = \ii (\D^+ (X,g))$ and $d\omega_{\alpha} = \ii (\D_{\alpha}^+ (X,g))$. The local index forms are related by $\ii (\D^+_{\alpha} (Z,g_Z)) = k \cdot \ii (\D^+(Z,g_Z))$ hence the transgressed classes can be chosen so that $\omega_{\alpha} = k\cdot \omega$. Now, subtracting $k$ copies of the first formula from the second gives the desired result. 
\end{proof}

We now wish to prove that a formula similar to \eqref{E:eq1} holds as well for certain metrics on $\zp$ which are conformally equivalent to the end-periodic metric $g_\infty$ but which do not need to be end-periodic themselves. The metrics in question will be of the form $g' = \sigma^2\,g_\infty$, where $\sigma: \zp \to \mathbb R$ is a positive smooth function such that 
\begin{enumerate}
\item[(a)] the scalar curvature of $g'$ is uniformly positive on $\zp$, and 
\item[(b)] both $\sigma$ and $\sigma^{-1}$ are bounded functions on $\zp$.
\end{enumerate}
An example of such a function is the function $\sigma = u^{2/(n-2)}$, where $u$ is constructed in Lemma \ref{lem:main}.

\begin{proposition}\label{P:two}
Let $\zp$ be an end-periodic manifold whose end is modeled on the infinite cyclic cover of $X$. Then, for any metric $g' = \sigma^2\,g_\infty$ as above, 
\begin{equation}\label{E:eq2}
\tw_{\rep}(X, \D^+)\, =\, k\cdot \ind \D^+ (\zp,g')\, -\, \ind \D^+_{\alpha} (\zp,g').
\end{equation}
\end{proposition}

\begin{proof}
First note that the metric $g'$ has uniformly positive scalar curvature, therefore, the operators $\D^+ (\zp,g')$ and $\D^+_{\alpha} (\zp,g')$ are uniformly invertible at infinity so their $L^2$ closures are Fredholm. The Dirac operators on $\zp$ corresponding to the metrics $g_\infty$ and $g' = \sigma^2\,g_\infty$ are related by the formulas
\[
\D^+(\zp,g') = \sigma^{-(n+1)/2}\circ\D^+(\zp,g_\infty)\circ\sigma^{(n-1)/2},
\]
\[
\D^+_{\alpha}(\zp,g') = \sigma^{-(n+1)/2}\circ\D^+_{\alpha}(\zp,g_\infty)\circ\sigma^{(n-1)/2}.
\]
It follows that, for any $L^2$ spinor $\phi$ in the kernel of $\D^+ (\zp,g_\infty)$, the spinor $\phi' = \sigma^{-(n-1)/2}\phi$ is in the kernel of $\D^+(\zp,g')$ and, moreover,
\[
\|\phi'\|^2_{L^2(\zp,g')}  = \int_{\zp} |\phi'|^2 d\vol_{g'} = \int_{\zp} \sigma |\phi|^2 d\vol_{g_\infty}\;\le\;C\,\|\phi\|^2_{L^2(\zp,g_\infty)}
\]
for some constant $C > 0$. Since this construction is reversible, it establishes an isomorphism between the kernels of $\D^+ (\zp,g_\infty)$ and $\D^+ (\zp,g')$. A similar argument applied to the cokernel of $\D^+ (\zp,g_\infty)$, and then to the kernel and cokernel of $\D^+_{\alpha} (\zp,g_\infty)$, establishes the equalities $\ind \D^+ (\zp,g') = \ind \D^+ (\zp,g_\infty)$ and $\ind \D^+_{\alpha} (\zp,g') = \ind \D^+_{\alpha}(\zp,g_\infty)$. The statement now follows from Proposition \ref{P:one}.
\end{proof}


\section{Cobordisms of positive scalar curvature}\label{S:bordisms}
We begin by recalling the definition of the bordism group $\Omega^{\spin,+}_n (BG)$. Given a discrete group $G$, consider the triples $(X,g,f)$ consisting of a closed oriented spin manifold $X$ of dimension $n$, a positive scalar curvature metric $g$ on $X$, and a continuous map $f: X \to BG$. Two triples $(X_0, g_0, f_0)$ and $(X_1, g_1, f_1)$ represent the same class in $\Omega^{\spin,+}_n (BG)$ if there is a spin cobordism $Z$ between $X_0$ and $X_1$ which admits a positive scalar curvature metric $g_Z$ such that $g_Z = g_i + dt^2$ in collar neighborhoods of $X_i$, $i = 0, 1$, and a continuous map $Z \to BG$ extending the maps $X_0 \to BG$ and $X_1 \to BG$. 

To get the statement about the fundamental group in Theorem~\ref{T:one}, we need the following lemma, well-known to experts.
\begin{lemma}\label{L:pi1}
Let $n \ge 4$ and let $G$ be a finite group. Then any class in  $\Omega^{\spin,+}_n (BG)$ is represented by a triple $(X,g,f)$ where $f_*:\pi_1(X) \to \pi_1(BG) = G$ is an isomorphism.
\end{lemma}
\begin{proof}
First note that we may assume that $f_*$ is surjective. To see this, choose a finite generating set $\{\gamma_1,\ldots,\gamma_k\}$ for $\Gamma$. By~\cite{gajer:cobordism}, the spin cobordism $W$ obtained by adding $k$ $1$-handles to $X \times I$ has a psc metric extending the one on $X$.  Evidently, the map $f:X \to BG$ extends over $W$, inducing a map on $\pi_1(W) \cong \pi_1(X) \ast F\langle g_1,\ldots,g_k\rangle \to G$ that sends the  generators $g_j$ to $\gamma_j$. Let $(X_1,g_1,f_1)$ denote the new boundary component of $W$; then $\pi_1(X_1) \to \pi_1(W)$ is an isomorphism and so $(f_1)_*$ is a surjection.

Let $K$ denote the kernel of $(f_1)_*$. Since $G$ is finite, $K$ is of finite index in $\pi_1(X_1)$ and hence is finitely generated. Hence one can do a further surgery on circles representing these elements, preserving the spin structure, and changing the fundamental group to $G$. Since $n\ge 4$, the  circles along which the surgeries are done have codimension at least $3$. Using~\cite{gajer:cobordism} once more, the cobordism gotten by adding $2$-handles to $X_1 \times I$ is a spin cobordism with a psc metric and with a map to $B\Gamma$ inducing an isomorphism on the fundamental group. The new boundary component of this cobordism is the desired representative of the original class in $\Omega^{\spin,+}_n (B\Gamma)$.
\end{proof}
From now on, we will assume that $3 \leq n \leq 6$. Let $(X_0,g_0)$ and $(X_1,g_1)$ be $n$--dimensional closed Riemannian spin manifolds of positive scalar curvature, and fix a choice of primitive cohomology classes $\gamma_i \in H^1(X_i; \Z)$, $i = 0,1$. According to \cite{SY}, there exist smoothly embedded volume minimizing hypersurfaces $Y_i\subset X_i$ which are Poincar\'e dual to $\gamma_i$ and which admit positive scalar curvature metrics $g_{Y_i}$ in the same conformal class as the metrics $g_i|_{Y_i}$ obtained by restriction. 

Cutting $X_i$ open along $Y_i$, where $i = 0,1$, we obtain cobordisms $W_i$ with boundary $\partial W_i = Y_i^- \cup Y_i^+$ as in Section \ref{S:periodic}. Note for future use that the infinite cyclic covers $\widetilde X_i \to X_i$ can be written as the infinite unions
\[
\widetilde X_i \,=\, \bigcup_{j\in\Z}\; W_i
\]
with appropriate identifications of the boundary components, and denote the respective `half-infinite' unions by 
\[
\widetilde X_i^-\, =\, \bigcup_{j\leq0}\; W_i \quad\text{and}\quad  
\widetilde X_i^+\, =\, \bigcup_{j\geq0}\; W_i.
\]

Let us now assume that $(X_0,g_0,\gamma_0)$ and $(X_1,g_1,\gamma_1)$ represent the same class in the psc-bordism group $\Omega^{\spin,+}_n(S^1)$. Then, according to \cite[Theorem 5]{BK}, there exists a spin cobordism $Z$ from $Y_0$ to $Y_1$ with a positive scalar curvature metric $g_Z$ which is a product metric $g_Z = dt^2 + g_{Y_i}$ near $Y_i$, $i = 0,1$. The metric $g_Z$ does not extend in any obvious way to a positive scalar curvature metric on the end-periodic manifold 
\begin{equation}\label{E:ZP}
\zp\, =\, \widetilde X_0^-\,\cup\, Z\, \cup\, \widetilde X_1^+.
\end{equation}
To construct such a metric, we proceed as follows. Note that $\zp$ has two ends, one of which, $\widetilde X_1^+$, is attached to the component $Y_1$ of the boundary of $Z$, and the other, $\widetilde X_0^-$, to the component $Y_0$ with reversed orientation. Correspondingly, we will write $Y_0^-$ and $Y_1^-$ for the boundaries of $\widetilde X_0^-$ and $\widetilde X_1^+$. We will use a similar convention for the boundary components of $W_i$ so the `$-$' component will be the one closer to the compact piece. Using \cite{AB}, replace the metric $g_i$ on $W_i$, $i = 0, 1$, with a new metric $g_i'$ satisfying the following conditions\,:
\begin{enumerate}
\item $g_i'=g_i$ away from a neighborhood of $Y_i^-$,
\item $g_i'=dt^2 + g_{Y_i}$ near $Y_i^-$, and
\item the conformal class $(W_i,[g_i'])$ is Yamabe positive; cf. Section \ref{sec:conformallaplacian}.
\end{enumerate}

Note that the scalar curvature of $g_i'$ may be negative. Using the product structure of $g_i'$, equip $\zp$ with the smooth metric
\smallskip
\[
g_\infty=
\begin{cases}
\; g_0&\text{ on }\quad\bigcup_{j\leq-1} W_0\\
\; g_0'&\text{ on }\quad\bigcup_{j=0}\; W_0\\
\; g_Z&\text{ on }\quad\, Z\\
\; g_1'&\text{ on }\quad\bigcup_{j=0}\; W_1\\
\; g_1&\text{ on }\quad\bigcup_{j\geq1}\; W_1.
\end{cases}
\]
\smallskip\noindent

The gluing near $Z$ can be visualized in Figure \ref{F:metrics} below.\\[1ex]
\begin{figure}[htb]
\labellist
\small\hair 2pt
\pinlabel {$W_0$} [ ] at 100 48
\pinlabel {$W_1$} [ ] at 430 48
\pinlabel {$Y_0^+$} [ ] at  -5 75
\pinlabel {$Y_0^-$} [ ] at 225 75
\pinlabel {$Y_1^-$} [ ] at 305 75
\pinlabel {$Y_1^+$} [ ] at 535 75
\pinlabel {$g_0$} [ ] at 100 99
\pinlabel {$g_1$} [ ] at 420 99
\pinlabel {$dt^2+ g_{Y_0}$} [ ] at 182 101
\pinlabel {$dt^2+ g_{Y_1}$} [ ] at 343 101
 
\pinlabel {$Z$} [ ] at 265 200
\pinlabel {$Y_0$} [ ] at  150 232
\pinlabel {$Y_1$} [ ] at 375 232
\pinlabel {$dt^2+ g_{Y_0}$} [ ] at 188 256
\pinlabel {$dt^2+ g_{Y_1}$} [ ] at 343 256

\endlabellist
\centering
\includegraphics[scale=0.75]{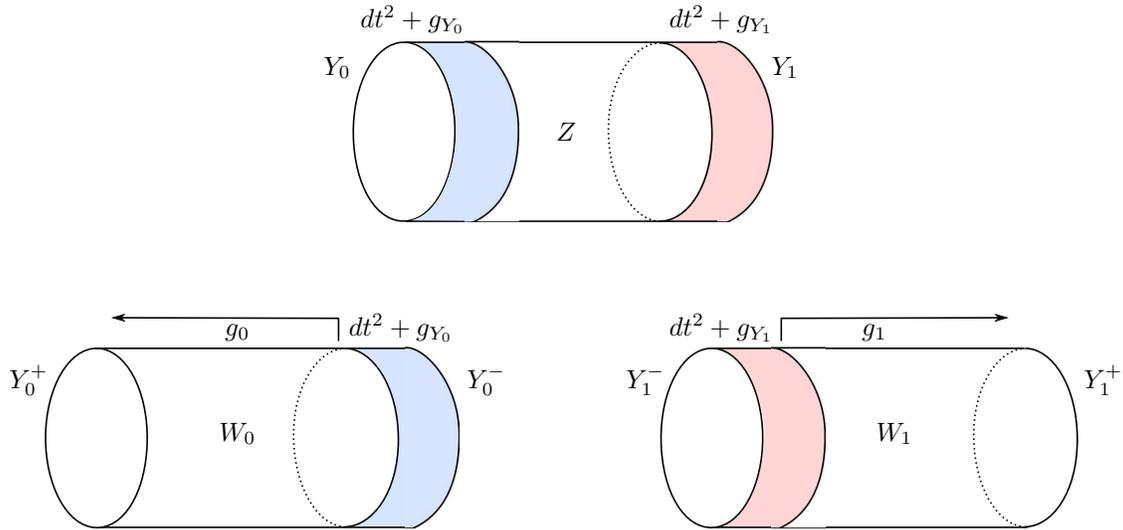}
\caption{Initial metric $g_\infty$}
\label{F:metrics}
\end{figure}

In what follows, we will use the notations
\[
X=X_0\sqcup X_1,\quad \widetilde X = \widetilde X_0\sqcup\widetilde X_1,\quad
W=W_0\sqcup W_1,\quad Y=Y_0\sqcup Y_1
\]
and, for the two metrics on $W$,
\[
g_W = g_0\sqcup g_1\quad\text{and}\quad g'_W = g'_0\sqcup g'_1.
\]
We will also fix a smooth function $f: \zp \to \R$ such that $f(\tau(x)) = x + 1$, where $\tau$ is the covering translation on either end, $f$ is constant on each copy of $Y$, $f^{-1} ([-1,0]) = Z$, and $f^{-1}([j,j+1])$ is the $j$th copy of $W$. 
The following is our main technical result, which will be proved in the next section. It roughly states that, in spite of the possible negative scalar curvature that $g_i'$ introduces, the non-compact manifold $(\zp,g_\infty)$ admits a well-controlled conformal change to a uniformly positive scalar curvature metric.

\begin{lemma}\label{lem:main}
There exists a positive function $u: \zp \to \R$ such that
\begin{enumerate}
\item $|u-1|\, \leq\, Ce^{-Bf}$ for some constants $B > 0$ and $C > 0$, and
\item the scalar curvature of the metric $u^{4/(n-2)}g_\infty$ is bounded from below by a positive constant.
\end{enumerate} 
\end{lemma} 

With Lemma \ref{lem:main} in place, we can complete the proof of Theorem \ref{T:two}. Assume that $(X_0, g_0, \gamma_0)$ and $(X_1, g_1, \gamma_1)$ represent the same class in the group $\Omega^{\spin,+}_n (S^1)$, and that the representations $\alpha_0: \pi_1 (X_0) \to U(k)$ and $\alpha_1: \pi_1 (X_1) \to U(k)$ extend to a representation of the fundamental group of the psc cobordism. Then the end-periodic manifold $\eqref{E:ZP}$ admits a uniformly positive scalar curvature metric $g'$ to which Proposition \ref{P:two} applies. This results in the formula 
\[
\tw_{\rep}(X_1, \D^+) - \tw_{\rep}(X_0, \D^+)\, =\, k\cdot \ind \D^+ (\zp,g')\, -\, \ind \D^+_{\alpha} (\zp,g').
\]
Since the metric $g'$ has positive scalar curvature, both indices on the right hand side of this formula must vanish, leading us to the conclusion that
\[
\tw_{\rep}(X_0, \D^+) \, =\, \tw_{\rep}(X_1, \D^+).
\]


\section{Proof of the main lemma}\label{S:laplace}
Our proof of Lemma \ref{lem:main} will be a modification of the argument from \cite[Proposition 4.6]{CM} dealing with a cylindrical Yamabe problem. We begin by recalling some basic facts about the conformal Laplacian on compact manifolds with boundary which will be essential in our proof.


\subsection{Conformal Laplacian on compact manifolds with boundary} \label{sec:conformallaplacian}
Let $(M,g)$ be a compact oriented $n$-dimensional Riemannian manifold with non-empty boundary $\partial M$. Denote by $\nu$ the outward normal vector to $\partial M$, by $R_g$ the scalar curvature of $M$, and by $H_g$ the mean curvature of $\partial M$. Consider the following pair of operators acting on $C^\infty(M)$:
\begin{equation}\label{eq:conformallaplacian}
\begin{cases}
\; L_g =-\Delta_g+c_n R_g&\quad \text{in $M$}, \\
\; B_g =\;\partial_\nu+2 c_nH_g&\quad \text{on $\partial M$},
\end{cases}
\end{equation}
where $c_n=(n-2)/(4(n-1))$. These operators describe the change in the scalar and boundary mean curvatures under a conformal change of metric: a standard computation shows that, for a positive function $\phi\in C^\infty(M)$, the scalar and boundary mean curvatures of the metric $\tilde g = \phi^{4/(n-2)}g$ are given by the formulas
\begin{equation}\label{eq:conformalchange}
\left\{
\begin{array}{lcll}
  R_{\tilde{g}} &=& c_n^{-1} \phi^{-(n+2)/(n-2)} \, L_{g} \phi & \mbox{in $M$}, \\
   H_{\tilde{g}} & = &
  \frac{1}{2}\, c_n^{-1}\phi^{-n/(n-2)} \, B_{g} \phi &
  \mbox{on $\partial M$}.
\end{array}
\right.
\end{equation}

Associated with the operators \eqref{eq:conformallaplacian} and every real--valued function $0 \neq \phi \in C^\infty(M)$ is the Rayleigh quotient
\[
\ddfrac{Q_g(\phi)}{\int_M\phi^2\,d\mu}
\]
where 
\begin{equation}\label{eq:Rayleigh2}
Q_g(\phi)=\int_M\left(|\nabla\phi|^2+c_nR_{g}\phi^2\right)d\mu+2c_n\int_{\partial M}H_{g}\phi^2d\sigma
\end{equation}
and $d\mu$ and $d\sigma$ denote the volume forms associated to $g$ and $g|_{\partial M}$, respectively. According to the standard elliptic theory,
\begin{equation}\label{eq:Rayleigh1}
\lambda\; = \inf_{0\neq \phi \in H^1(M)} \ddfrac{Q_g(\phi)}{\int_M\phi^2\,d\mu}
\end{equation}
is the principal eigenvalue of the boundary value problem $(L_g,B_g)$ and there is a positive function $\phi\in C^\infty(M)$, unique up to scalar multiplication, such that
\begin{equation}\label{eq:eigenvalue}
\left\{
\begin{array}{lcll}
  L_{g}\phi &= & \lambda\phi&\quad \text{in $M$},
  \\
  B_{g}\phi &= & 0 &\quad\text{on $\partial M$}.
\end{array}
\right.
\end{equation}  
This eigenvalue problem was first studied by Escobar \cite{E1} in the context of the Yamabe problem on manifolds with boundary.  

Let $\phi$ be a positive solution of (\ref{eq:eigenvalue}) and consider the metric $\tilde{g}=\phi^{4/(n-1)}g$. It follows from (\ref{eq:conformalchange}) that the boundary mean curvature of the metric $\tilde g$ vanishes, and that the scalar curvature of $\tilde g$ has a constant sign agreeing with the sign of $\lambda$. In particular, the sign of $\lambda$ is an invariant of the conformal class $[g]$ of metric $g$. A conformal manifold $(M,[g])$ is called {\emph{Yamabe positive, negative, or null}} if $\lambda$ is respectively positive, negative, or zero; see Escobar \cite{E1,E3}.


\subsection{Preliminary eigenvalue estimates}
For each positive integer $k$, denote by $Z_k$ the compact submanifold of $\zp$ obtained by attaching only the first $k$ copies of $W$ to $Z$. Let $g_k$ be the metric $g_\infty$ restricted to $Z_k$, and $\lambda(k)$ the principal eigenvalue of the boundary value problem \eqref{eq:eigenvalue} for the operators $(L_{g_k}, B_{g_k})$. Since the hypersurfaces $Y_0$ and $Y_1$ were chosen to be minimal, each $Z_k$ has  vanishing boundary mean curvature, and the Rayleigh quotient (\ref{eq:Rayleigh1}) for $\lambda(k)$ takes the form
\begin{equation}\label{eq:Rayleigh3}
\lambda(k)\; =\inf_{0\neq \phi \in H^1(Z_k)}\;
	\ddfrac{\int_{Z_k}\,(|\nabla\phi|^2+c_{n}R_{g_k}\phi^2)\,d\mu}{\int_{Z_k}\phi^2\,d\mu}
\end{equation}

\begin{proposition}\label{prop:eigenvalues}
There are positive constants $C_1$ and $C_2$ depending only on the metrics $g_Z$, $g_W$, and $g'_W$ such that, for all positive $k$, 
\[
C_1\; \le\; \lambda(k)\; \le\; C_2.
\]
\end{proposition}
\begin{proof}
The conformal manifolds $(Z,[g_Z])$, $(W,[g_W])$, and $(W,[g'_W])$ are all Yamabe positive by construction. As mentioned in Section \ref{sec:conformallaplacian}, this is equivalent to the positivity of the principal eigenvalue of the boundary value problem \eqref{eq:conformallaplacian}.
Since the boundaries of $(Z,g_Z)$, $(W,g_W)$, and $(W,g'_W)$ are all minimal, their mean curvatures vanish and the boundary conditions in \eqref{eq:conformallaplacian} reduce to the Neumann boundary conditions, $\partial_\nu\phi = 0$. This implies that the principal Neumann eigenvalues $\lambda (L_{g_Z})$, $\lambda (L_{g_W})$, and $\lambda (L_{g'_W})$ are all positive. 

Let us consider the principal Neumann eigenvalue $\lambda (1) = \lambda(L_{g_{W_1}})$ of the operator $L_{g_{W_1}}$. Since $W_1$ is split into $(Z,g_Z)$ and $(W,g'_W)$,
\[
\lambda(L_{g_{W_1}})=\inf\left\{\,Q_{g_{W_1}}(\phi)\; |\; 0\neq\phi\in H^1(W_1),\; ||\phi||^2_{L^2}=1\right\}  
\]
can be estimated from below by
\begin{multline}\notag
 \inf \Big\{Q_{g_Z}(\phi_1) + Q_{g'_W}(\phi_2)\; | \; \phi_1\in H^1(Z),\; \phi_2\in H^1(W),\; ||\phi_1||^2_{L^2}+||\phi_2||^2_{L^2} = 1\Big\} = \\ 
 \inf_{a\in[0,1]} \inf \Big\{ a\cdot \frac{Q_{g_Z}(\phi_1)}{||\phi_1||^2_{L^2}}\, +\, (1-a)\cdot \frac{Q_{g'_{W}}(\phi_2)}{||\phi_2||^2_{L^2}}\;\; \Big| \; \phi_1\in H^1(Z), \; \phi_2\in H^1(W), \\ ||\phi_1||^2_{L^2}=a,\; ||\phi_2||^2_{L^2}=1-a \Big\},
\end{multline}
which is in turn estimated from below by
\[
\inf_{a\in[0,1]} \Big\{a\cdot \lambda(L_{g_Z}) + (1-a)\cdot \lambda(L_{g'_{W}})\Big\}\; =\; \min\Big\{\lambda(L_{g_Z}),\, \lambda(L_{g'_{W}})\Big\} > 0.
\]
By splitting $Z_k$ into $(Z_{k-1},g_{k-1})$ and $(W,g_W)$ and proceeding inductively, one can use the above argument to show that 
\[
\lambda(k)\;\geq\;\min\Big\{\lambda(L_{g_Z}),\, \lambda(L_{g_W}),\, \lambda(L_{g'_W})\Big\} > 0
\] 
for all positive $k$. This gives the desired lower bound on the eigenvalues $\lambda(k)$. To obtain the upper bound, choose the constant test function $\phi = 1$ in the Rayleigh quotient to obtain
\[
\lambda(k)\; \leq\; \frac 1 {\operatorname{vol}(Z_k)} \cdot \int_{Z_k} c_n R_{g_k}\,d\mu\; \leq\; c_n\cdot \sup_{Z_k}\; (R_{g_k}).
\]
Since the scalar curvature of $g_\infty$ is uniformly bounded from above, this gives the desired upper bound on $\lambda(k)$.
\end{proof}


\subsection{Strategy of the proof} 
The function $u: \zp \to \mathbb R$ whose existence is claimed in Lemma \ref{lem:main} will be obtained as a solution of the equation $L_{g_\infty} (u) = h$, where $h: \zp \to \R$ is a smooth positive function such that
\begin{enumerate}
\item $h = c_n R_{g_\infty}$ on $\zp \setminus Z_3$ and 
\item $h\geq |\,c_nR_{g_\infty}|$ on $Z_3$.
\end{enumerate}
To solve this equation, we will first solve the equation $L_{g_{\infty}} (v) = \tilde h$ for the function $\tilde h = h - c_n R_{g_\infty}$, which is positive and compactly supported in $Z_3$, and then let $u=1+v$. The equation $L_{g_{\infty}} (v) = \tilde h$ will be solved by the barrier method.


\subsection{The barrier method} 
Let $C_1 > 0$ be the constant from Proposition \ref{prop:eigenvalues} and introduce the constant
\[
\bar\lambda\,=\,\frac 12\, \min\left\{\, C_1\, ,\;\inf_{Z_\infty\setminus Z_3} c_n R_{g_{\infty}}\right\}.
\]

\smallskip\noindent
Note that $\bar\lambda$ is positive, and also that the function $c_n R_{g_\infty}-\bar\lambda$ is positive on $Z_\infty\setminus Z_3$. The two propositions that follow provide two ingredients for the barrier construction.


\begin{proposition}\label{prop:barrier1}
For any point $x_0 \in Z$ there exists a smooth positive function $w: \zp \to \R$ such that
\[
\begin{cases}
\; (L_{g_\infty}-\bar\lambda)(w) = 0& \text{in}\;\; \zp, \;\; \text{and} \\
\; w(x_0)= 1.&
\end{cases}
\]
\end{proposition}

\begin{proof}
It follows from Proposition \ref{prop:eigenvalues} that $0 < \bar\lambda < \lambda(k)$ for all positive $k$,  hence the operators $L_{g_k} - \bar\lambda$ with the Neumann boundary conditions are all invertible modulo constant functions. Since the principal eigenvalue of an elliptic operator with the Dirichlet boundary condition is greater than or equal to the principal eigenvalue of the same operator with the Neumann boundary condition, it follows that $L_{g_k} - \bar\lambda$ with the Dirichlet boundary condition is also invertible. The rest of the proof follows the argument of \cite[Theorem 2.1]{S}. For each positive $k$, consider a smooth solution $w_k'$ of the problem
\[
\begin{cases}
\; (L_{g_\infty}-\bar\lambda)(w_k')=0&\text{ in }Z_k\\
\; w'_k = 0&\text{ on }\partial Z_k
\end{cases}
\]
which is positive in the interior of $Z_k$. Define $w_k = w'_k/w'_k(x_0)$. Then a standard argument using the Harnack inequality and Schauder estimates shows that there is subsequence of $w_k$ which converges to a positive function $w\in C^\infty (\zp)$ on compact subsets of $\zp$ in the $C^{k,\alpha}$ topology. Proposition \ref{prop:barrier1} follows.
\end{proof}

\begin{proposition}\label{prop:barrier2}
There are positive constants $B$ and $C_3$ such that the function $\phi=e^{-Bf}: \zp \to \R$ satisfies the inequality 
\[
L_{g_\infty}(\phi)\; \geq\; C_3\,  \phi\quad \text {on}\quad \zp \setminus Z_3.
\]
\end{proposition}

\begin{proof}
Let $B > 0$ be an arbitrary constant, to be determined later, and $\phi = e^{-Bf}$. Then
\begin{align*}
\Delta_{g_\infty}\phi\, & = B^2|\nabla f|^2e^{-Bf} - (\Delta_{g_\infty} f)\, Be^{-Bf} \\
& \leq B^2\max(|\nabla f|^2)e^{-Bf} + \max(|\Delta_{g_\infty} f|)\, Be^{-Bf} \\ 
& = B(B\max(|\nabla f|^2)+\max(|\Delta_{g_\infty} f|))\,\phi.
\end{align*}
This allows us to make the estimate
\begin{align}
L_{g_\infty}\phi = & -\Delta_{g_\infty}\phi+c_nR_{g_\infty}\phi \notag \\
& \geq (-B(B\max(|\nabla f|^2)+\max(|\Delta_{g_\infty} f|)) +c_n\min(R_{g_\infty}))\,\phi.\label{eq:barrierest}
\end{align}
From the construction of $f$, we know that $\max(|\nabla f|^2)$ and $\max(\Delta_{g_\infty} f)$ are both finite. Moreover, the minimum of $R_{g_\infty}$ on $\zp \setminus Z_3$ is positive. Combining these facts with inequality (\ref{eq:barrierest}), we conclude that $B > 0$ can be chosen small enough so that
\[
L_{g_\infty}\phi\; \geq\; C_3\,\phi\quad\text{on}\quad \zp \setminus Z_3, 
\]
where $C_3 = \frac 1 2\, c_n\min(R_{g_\infty})$.
\end{proof}

The barrier for the equation $L_{g_\infty}(v) = \tilde h$ is now constructed by piecing together the functions $w$ and $\phi$.

\begin{proposition}\label{prop:barrier3}
There are continuous functions $\overline\phi$, $\underline\phi: \zp \to \R$ which satisfy the inequalities 
\[
L_{g_\infty} (\underline\phi)\; \leq\; \tilde h\quad\text{and}\quad L_{g_\infty}(\overline\phi)\; \geq\; \tilde h
\]
weakly. Moreover, there is a positive constant $C_4$ such that $|\underline\phi|$, $|\overline\phi|\;\leq\; C_4\,\phi$ on $\zp \setminus Z_3$. 
\end{proposition}

\begin{proof}
We start by choosing constants $\alpha$, $\beta>0$ such that $\alpha w < \beta\phi$ on $\partial Z_3$ and both $L_{g_\infty}(\alpha w) \geq \tilde h$ on $\zp$ and $L_{g_\infty}(\beta\phi) \geq \tilde h$ on $\zp \setminus Z_3$. 

{\bf{Case 1:}} There exists a positive integer $k_0\geq4$ such that $\beta\phi\leq \alpha w$ on $Z_{k_0+1}\setminus Z_{k_0}$.
In this case, we define
\[
\overline\phi(x):=
\begin{cases}
\; \alpha w(x)&\text{if $x\in Z_3$}\\
\; \min(\alpha w(x),\beta\phi(x))&\text{if $x\in Z_{k_0}\setminus Z_3$}\\
\; \beta\phi(x)&\text{if $x\in \zp\setminus Z_{k_0}$}
\end{cases}
\]
and let $\underline\phi = -\overline\phi$. The functions $\overline\phi$ and $\underline\phi$ are continuous super- and sub-solutions of the equation $L_{g_\infty}(\cdot) = \tilde h$ on $\zp$. This follows 
from the fact that the minimum (resp. maximum) of two super-solutions (resp. sub-solutions) is again a super-solution (resp. sub-solution); see for instance \cite[Section 2.8]{GT}.

{\bf{Case 2:}} The assumption of Case 1 does not hold. In other words, for every integer $k\geq 4$, there is a point $x_k \in Z_{k+1}\setminus Z_k$ such that $\alpha w(x_k)<\beta\phi(x_k)$ and, in particular, $\alpha w(x_k)<\beta e^{-Bk}$.

Since the the distance between $x_k$ and the boundary components $\partial Z_k$ and $\partial Z_{k+1}$ is less than the diameter of $W$ for all $k$, the Harnack inequality for $w$ provides us with a constant $D > 0$ such that 
\[
\alpha w(x)\leq D\alpha w(x_k)<D\beta e^{-Bk}
\]
for all $x\in Z_{k+1}\setminus Z_{k}$. It now follows that there is a constant $C_4 > 0$ such that $\alpha w \leq C_4 \phi$ on $\zp \setminus Z_4$. In this case, we simply define $\overline\phi=\alpha w$ and $\underline\phi = -\overline\phi$.
\end{proof}

Having constructed the barrier, we can put it to use proving the following existence result.

\begin{proposition}\label{prop:barrier4}
There exists a smooth function $v: \zp \to \R$ such that
\begin{enumerate}
\item $L_{g_\infty}(v) = \tilde h$ on $\zp$, and
\item $\underline\phi\; \leq\, v\,\leq\, \overline\phi$ on $\zp$.
\end{enumerate}
\end{proposition}

\begin{proof}
For each $k\geq3$, we find a solution $v_k: Z_k \to \R$ to the problem
\[
\begin{cases}
\; L_{g_\infty}(v_k) = \tilde h&\text{in $W_k$} \\
\; v_k = \underline\phi&\text{on $\partial W_k$}.
\end{cases}
\]
Since $\underline\phi$ is a subsolution, one can apply the maximum principle to $v_k-\underline\phi|_{Z_k}$ to conclude that $v_k\geq\underline\phi|_{Z_k}$. Likewise, one concludes that $v_k \leq \overline\phi|_{Z_k}$. A standard argument now shows that $v_k$ converges in the $C^{2,\alpha}$ topology to a smooth function $v$ on $\zp$ which satisfies $L_{g_\infty}(v) = \tilde h$.
\end{proof}


\subsection{Finishing the proof}
We are now ready to finish the proof of Lemma \ref{lem:main}. The following proposition ensures that the function $u = 1+v$, where $v$ is the function of Proposition \ref{prop:barrier4}, is positive and hence can serve as a conformal factor.

\begin{proposition}
The function $u = 1+v$ solves the equation $L_{g_\infty}(u) = h$. Moreover, $u > 0$ on $\zp$.
\end{proposition}

\begin{proof}
The following argument is reproduced from the proof of \cite[Proposition 4.6]{CM}. From the exponential decay of the function $v$ in Proposition \ref{prop:barrier4} it is clear that $u > 0$ on $\zp\setminus Z_k$
for some sufficiently large $k > 3$. Let $\phi_0 > 0$ be a positive eigenfunction corresponding to the principal eigenvalue $\lambda^D(k)$ of the operator $L_{g_{k}}$ on $Z_k$ with the Dirichlet boundary condition, that is, 
\[
\begin{cases}
\; L_{g_k}(\phi_0) = \lambda^D(k)\,\phi_0&\text{in $Z_k$} \\
\; \phi_0 = 0&\text{on $\partial Z_k$}.
\end{cases}
\]
Then one may consider the function $u/\phi_0$ and calculate as in \cite[Proposition 4.6]{CM} that
\[
\Delta_{g_k}\left(\frac{u}{\phi_0}\right) + \frac{2}{\phi_0}\left\langle\nabla\phi_0,\nabla\left(\frac{u}{\phi_0}\right)\right\rangle - \lambda^D(k)\left(\frac{u}{\phi_0}\right) = - \frac{h}{\phi_0}<0.
\]
It follows that at an interior minimum in $Z_k$, the function $u/\phi_0$ is positive; the minimum of $u/\phi_0$ must lie in the interior because $\phi_0 = 0$ on $\partial Z_k$.
\end{proof}

It follows from formula \eqref{eq:conformalchange} that the metric $u^{4/(n-2)} g_\infty$ has positive scalar curvature. Since $u - 1 = v$ decays exponentially by Proposition \ref{prop:barrier4}, the proof of Lemma \ref{lem:main} is complete.


\section{Proof of Theorem \ref{T:one}}\label{S:reps}
We will prove Theorem \ref{T:one} by constructing connected psc manifolds $M$ with fundamental group $\Gamma$ and representations $\alpha$ so that the invariants $\tw_\alpha$ take on infinitely many distinct values. Note that the connectedness of $M$ is an essential part of Theorem \ref{T:one}; disconnected examples could be constructed much more easily. 

In dimension six, it follows from \cite[Theorem 0.1]{botvinnik-gilkey:spaceforms} that, for each finite group $G$ with $r_5(G) > 0$, there are closed connected spin manifolds $Y^5$ which admit infinitely many psc metrics $g_i$ distinguished up to bordism by invariants $\tw_\alpha (Y,g_i)$, where $\alpha$ is a unitary representation of $G$. By Lemma~\ref{L:pi1} we may assume that $\pi_1(Y) \cong G$. The pull back via projection $\Z \times G \to G$ gives rise to a unitary representation of $\Z \times G$ called again $\alpha$. Since the periodic $\eta$-invariants of~\cite{MRS3} satisfy $\eta (S^1 \times Y,\D^+_{\alpha}) = \eta_{\alpha}(Y)$ when $S^1 \times Y$ is given a product metric, Theorem~\ref{T:two} implies that the periodic invariants $\tw_\alpha(S^1 \times Y,dt^2 + g_i)$ are all distinct up to bordism. Note that this argument gives the {\em a priori} stronger conclusion that all of these non-bordant metrics live on the same manifold.

This argument does not work in dimension $4$, because $3$-dimensional space forms support a unique psc metric, up to isotopy~\cite{marques:psc}. One issue is that the standard technique for pushing a psc metric across a cobordism, which is crucial to the constructions in~\cite{botvinnik-gilkey:bordism,botvinnik-gilkey:spaceforms}, does not work for pushing a metric across a $4$-dimensional cobordism.  In general, one is not able to push a psc metric across a $5$-dimensional cobordism either.  However, Section 9.3 in~\cite{MRS3}  shows how to create psc-cobordisms between $4$-manifolds, at the expense of taking connected sums with some unknown number of copies of $S^2 \times S^2$.  

For a non-simply connected space form $S^3/G$, the proof of~\cite[Theorem 9.5]{MRS3}, sketched below, shows that for an appropriate $\alpha$ and any $N\geq 1$, there is a non-negative number $m_N$ such that 
$$
\tw_\alpha\left((S^1 \times (S^3/G)) \conn m_N \cdot (S^2 \times S^2)\right)
$$
takes on $N$ different values. By letting $N$ go to infinity, we see that $\Omega^{\spin,+}_{4}(S^1 \times BG)$ must be infinite. We note that, in contrast to the $6$-dimensional result, it is not clear if infinitely many non-bordant metrics could be supported on the same manifold.\qed

\begin{remark}
The hypothesis in the $6$-dimensional case that $Y$ admit a spin structure can presumably be omitted. This would involve extending the results of~\cite{MRS3} to include `twisted' spin structures as discussed in the introduction of~\cite{botvinnik-gilkey:spaceforms}. \\
\end{remark}

\begin{remark}
Because the metrics in the $6$-dimensional case are product metrics, we do not really need the end-periodic index theorem to prove Theorem \ref{T:one} in this case; the Atiyah--Patodi--Singer theorem~\cite{aps:II} (as extended in Proposition \ref{P:two}) would suffice. On the other hand, it does not seem possible to prove the $4$-dimensional case of Theorem \ref{T:one} without the use of~\cite{MRS3}.
\end{remark}

For the sake of completeness, we provide here a sketch of the construction of the metrics on the $4$-manifolds described in the proof above. Details can be found in~\cite[Section 9]{MRS3}

Write $Y= S^3/G$, and let $g$ be a psc metric descending from the round metric on $S^3$. Lemma 9.7 of~\cite{MRS3} provides a representation $\alpha: \pi_1(Y) \to U(k)$ for which the invariant $\tw_\alpha(Y,g)$ is non-zero. For any $n$,  the finiteness of the spin cobordism group $\Omega^{\spin}_3(B\pi_1(Y))$ and some additional topological arguments are used construct a spin cobordism $V_n$ with $\pi_1(V_n) \cong \pi_1(Y)$. There is an extension $\tilde{\alpha}$ of the representation $\alpha$ so that 
$$
\p\, (V_n,\tilde\alpha)\; =\; (Y,\alpha)\; -\; (nd +1) \cdot (Y,\alpha)
$$
Next, cross $V_n$ with a circle to obtain a cobordism $W_n$ from $S^1 \times  ((nd +1) \cdot  Y)$
(the lower part of the boundary) to $S^1 \times Y$. It can be assumed that this cobordism has a handlebody decomposition with handles of index only $2$ and $3$. Write $k_n$ for the number of $3$-handles.

The argument of~\cite{botvinnik-gilkey:bordism,botvinnik-gilkey:spaceforms} would at this point be to push the psc metric from the bottom end of $W_n$ to $S^1 \times Y$, but the presence of the $3$-handles  prohibits this, since they are attached along spheres of codimension $2$. However, the psc metric can be pushed across 
$$
S^1 \times  ( (nd +1) \cdot Y) \times I
$$
plus the $2$-handles. The upper boundary of this manifold is shown to be $(S^1 \times Y) \#\, k_n \cdot  (S^2 \times S^2)$, which therefore acquires a psc metric. By construction, all of these manifolds have fundamental group $\Z \times G$, and their $\tw_\alpha$ invariants grow linearly with $n$ and hence by Theorem~\ref{T:two}, their bordism classes are distinct.  By adding additional copies of $S^2 \times S^2$ to these manifolds, one can obtain diffeomorphic manifolds carrying an arbitrary number of these distinct bordism classes.\qed


\end{document}